\documentclass{article}
\usepackage{amsmath}
\usepackage{amssymb}
\usepackage{amsthm}
\usepackage{url}

\newtheorem{theorem}{Theorem}
\newtheorem{corollary}{Corollary}
\newtheorem{lemma}{Lemma}
\theoremstyle{definition}
\newtheorem{remark}{Remark}

\begin{document}

\title{Some formulas for numbers of line segments and lines in
a rectangular grid}

\author{Pentti Haukkanen\footnote{E-mail:
pentti.haukkanen@uta.fi}\,\, and Jorma K. Merikoski\footnote{E-mail:
jorma.merikoski@uta.fi}\\ School of Information Sciences\\ FI-33014
University of Tampere, Finland}
\date{}

\maketitle

\begin{abstract}
We present a formula for the number of line segments connecting $q+1$
points of an $n_1\times\dots\times n_k$ rectangular grid. As
corollaries, we obtain formulas for the number of lines through at
least $q$ points and, respectively, through exactly $q$ points of the
grid. The well-known case $k=2$ is so generalized. We also present
recursive formulas for these numbers assuming $k=2$, $n_1=n_2$. The
well-known case $q=2$ is so generalized.


\end{abstract}

\section{Introduction}

Let us consider a rectangular grid
\[
G(n_1,\dots,n_k)=
\{0,\dots,n_1-1\}\times\dots\times\{0,\dots,n_k-1\},
\]
where $k,n_1,\dots,n_k\ge 2$. Call its points \textit{gridpoints}.
Given $q\ge 2$, we say that a line is a \textit{$q$-gridline} if it
goes through exactly $q$ gridpoints. We write $l_q(n_1,\dots,n_k)$ for
the number of $q$-gridlines, and $l_{\ge q}(n_1,\dots,n_k)$ for the
number of gridlines through at least $q$ gridpoints. In other words,
$l_{\ge q}(n_1,\dots,n_k)$ is the sum of all $l_p(n_1,\dots,n_k)$'s
with $p\ge q$.

We also say that a line segment is a \textit{$q$-gridsegment} if its
endpoints and exactly $q-2$ interior points are gridpoints. Let
$c_q(n_1,\dots,n_k)$ denote the number of all $q$-gridsegments. (If
$q>2$, some of them may partially overlap.) In other words,
$c_q(n_1,\dots,n_k)$ is the number of line segments between such
gridpoints that are visible to each other through $q-2$ gridpoints.

Our first problem, discussed in Section~\ref{explicit}, is to find
formulas for these numbers. They are well-known~\cite{Mu1,Mu2,Sl} in
case of~$k=2$. We will see that Mustonen's~\cite{Mu1, Mu2} ideas work
also in the general case.

These ``explicit'' formulas may have theoretical value but their
practical value is small. They are computationally tedious and do not
tell much about the behaviour of the functions. This motivates to look
for recursive or asymptotic formulas. We do not discuss asymptotic
formulas here but refer to~\cite{EMHM} and note that we are pursuing
this matter further~\cite{HM}.

So our second problem is to find recursive formulas. We consider only
the case $k=2$, $n_1=n_2$ and denote $l_q(n)=l_q(n,n)$, $l_{\ge
q}(n)=l_{\ge q}(n,n)$, $c_q(n)=c_q(n,n)$. Mustonen~\cite{Mu1}
conjectured and Ernvall-Hyt\"onen et al.~\cite{EMHM} proved recursive
formulas for~$l_{\ge 2}(n)$. We will in Section~\ref{recursive} settle
this problem generally. In fact, Mustonen~\cite{Mu2} has already done
so, but, in our opinion, his results deserve publication also in a
journal.

We will complete our paper with remarks in Section~\ref{remarks}.

\section{Explicit formulas}\label{explicit}

\begin{theorem}\label{explthm}
For all $k,n_1,\dots,n_k\ge 2$, $q\ge 1$,
\[
c_{q+1}(n_1,\dots,n_k)=\frac{1}{2}f_q(n_1,\dots,n_k),
\]
where
\[
f_q(n_1,\dots,n_k)=
\sum_{\substack{-n_1<i_1<n_1\\\dots\\[0.5ex]-n_k<i_k<n_k\\
\gcd(i_1,\dots,i_k)=q}}(n_1-|i_1|)\cdots(n_k-|i_k|).
\]
\end{theorem}

\begin{proof}
Consider a line segment $c$ whose endpoints $A=(x_1,\dots,x_k)$ and
$B=(x_1+i_1,\dots,x_k+i_k)$ are gridpoints. Then $c$ is a
$(q+1)$-gridsegment if and only if also the points
\[
\Bigl(x_1+\frac{t}{q}i_1,\dots,x_k+\frac{t}{q}i_k\Bigr),\quad
t=1,\dots,q-1,
\]
are gridpoints and there are no other gridpoints in~$c$. This happens
if and only if $\gcd{(i_1,\dots,i_k)}=q$. Fix $i_1,\dots,i_k$ with the
properties
\begin{equation}\label{ij}
\gcd{(i_1,\dots,i_k)}=q \quad\text{and}\quad
-n_j<i_j<n_j,\quad
j=1,\dots,k.
\end{equation}
We can choose $x_j$ in $n_j-|i_j|$ ways, and so the number of all
possible choices of~$A$ is $ (n_1-|i_1|)\cdots(n_k-|i_k|). $ Summing
over all $i_1,\dots,i_k$ satisfying~(\ref{ij}), we
obtain~$f_q(n_1,\dots,n_k)$. We halve it, because $-i_1,\dots,-i_k$
lead to same gridsegments as $i_1,\dots,i_k$. So the theorem follows.
\end{proof}

\begin{corollary}\label{explcor1}
Let $k,n_1,\dots,n_k,q\ge 2$. Then
\[
l_{\ge
q}(n_1,\dots,n_k)=\frac{1}{2}(f_{q-1}(n_1,\dots,n_k)-
f_q(n_1,\dots,n_k)).
\]
\end{corollary}

\begin{proof}
Let $L$ be the set of $p$-gridlines over all $p\ge q$. Take $l\in L$.
It contains $p-q+1$ $q$-gridsegments. On the other hand, it contains
$p-q$ $(q+1)$-gridsegments. If $N(r,l)$ stands for the number of
$r$-gridsegments on~$l$, we therefore have $N(q,l)-N(q+1,l)=1$, and so
\begin{align*}
    l_{\ge q}(n_1,\dots,n_k)=\sum_{l\in L}1&=\sum_{l\in L}(N(q,l)-N(q+1,l)) \\
    &= \sum_{l\in L}N(q,l)-\sum_{l\in L}N(q+1,l) \\
    &= c_q(n_1,\dots,n_k)-c_{q+1}(n_1,\dots,n_k) \\
    &= \frac{1}{2}f_{q-1}(n_1,\dots,n_k)-\frac{1}{2}f_q(n_1,\dots,n_k).
\end{align*}
\end{proof}

\begin{corollary}\label{explcor2}
Let $k,n_1,\dots,n_k,q\ge 2$. Then
\[
l_q(n_1,\dots,n_k)=\frac{1}{2}(f_{q+1}(n_1,\dots,n_k)
-2f_q(n_1,\dots,n_k)+f_{q-1}(n_1,\dots,n_k)).
\]
\end{corollary}

\begin{proof}
Simply note that
\[
    l_q(n_1,\dots,n_k)=l_{\ge q}(n_1,\dots,n_k)-l_{\ge q+1}(n_1,\dots,n_k).
\]
\end{proof}

\section{Recursive formulas}\label{recursive}

Let us consider the function
\[
e_q(n)=\begin{cases}
    \phi(\frac{n-1}{q}) & \text{if }q\mid n-1,\\
    0 & \text{if }q\nmid n-1,
\end{cases}
\]
where $\phi$ is the Euler totient function. We present first the
following elementary lemma. If~$q=1$, then (\ref{sum1}) is trivial and
(\ref{sum2}) is well-known (e.g.,~\cite[Exercise~2.16]{Ap},
\cite[Lemma~1]{EMHM}).

\begin{lemma}
\label{sums} Let $n\ge 2$, $q\ge 1$. Then
\begin{equation}
\label{sum1}
\sum_{\substack{i=1\\(i,n)=q}}^n1=e_q(n+1).
\end{equation}
Furthermore,
\begin{equation}
\label{sum2}
\sum_{\substack{i=1\\(i,n)=q}}^ni=
\sum_{\substack{i=1\\(i,n)=q}}^n(n-i)=\frac{1}{2}ne_q(n+1)
\end{equation}
if $q\ne n$, and
\begin{equation}
\label{sum3}
\sum_{\substack{i=1\\(i,n)=n}}^ni=n,\quad
\sum_{\substack{i=1\\(i,n)=n}}^n(n-i)=0.
\end{equation}
\end{lemma}

\begin{proof}
If $q\nmid n$, then both sides of~(\ref{sum1}) and all sides
of~(\ref{sum2}) are zero. The case~$q\mid n$ remains; let $n=kq$. Then
\[
\sum_{\substack{i=1\\(i,n)=q}}^n1=
\sum_{\substack{i=1\\(i,k)=1}}^k1=\phi(k)=
\phi(\frac{n}{q})=e_q(n+1),
\]
and (\ref{sum1}) follows. To show~(\ref{sum2}), assume $k>1$. Because
\[
\sum_{\substack{i=1\\(i,n)=q}}^ni=
\sum_{\substack{i=1\\(n+1-i,n)=q}}^n(n+1-i)=
\sum_{\substack{i=0\\(n-i,n)=q}}^{n-1}(n-i)=
\sum_{\substack{i=1\\(i,n)=q}}^n(n-i),
\]
the first equality holds. Hence, by~(\ref{sum1}),
\[
2\sum_{\substack{i=1\\(i,n)=q}}^ni=
\sum_{\substack{i=1\\(i,n)=q}}^ni+
\sum_{\substack{i=1\\(i,n)=q}}^n(n-i)=
\sum_{\substack{i=1\\(i,n)=q}}^nn=
n\sum_{\substack{i=1\\(i,n)=q}}^n1=ne_q(n+1),
\]
and so also the second holds. The claim~(\ref{sum3}) is trivial.
\end{proof}

\begin{theorem}\label{recthm1}
Let $n\ge 2$, $q\ge 1$. Then
\begin{align}
    f_q(n) &= 2f_q(n-1,n)-f_q(n-1)+r_q(n), \label{fq1} \\
    f_q(n-1,n) &= 2f_q(n-1)-f_q(n-2,n-1)+s_q(n). \label{fq2}
\end{align}
Here $f_q(n)=f_q(n,n)$,
\begin{equation*}
r_q(n)=8(e_q(2)+e_q(3)+\dots+e_q(n))
\end{equation*}
and so
\begin{equation*}
r_q(n)-r_q(n-1)=8e_q(n),
\end{equation*}
and
\begin{equation*}
s_q(n)=2(n-1)e_q(n).
\end{equation*}
The initial values are $f_q(n)=f_q(n-1,n)=0$ for $n\le q$.
\end{theorem}

\begin{proof}
If $q>n-1$, then everything is zero. So we assume $q\le n-1$.

\medskip
\noindent\textit{Case 1.} $q<n-1$. Since
\begin{align}
    f_q(n)
    &= 4\Bigl[n(n-q)+\sum_{\substack{0<i,j<n\\(i,j)=q}}(n-i)(n-j)\Bigr] \notag \\
    &= 4\Bigl[n^2-qn+\sum_{\substack{0<i,j<n-1\\(i,j)=q}}(n-i)(n-j)
        +2\sum_{\substack{i=1\\(i,n-1)=q}}^{n-2}(n-i)\Bigr], \label{fqproof1}
\end{align}
\begin{multline}\label{fqproof2}
    f_q(n-1,n) \\
    \begin{aligned}
    &= 2\Bigl[(n-1)(n-q)+(n-1-q)n+2\sum_{\substack{0<i<n-1\\0<j<n\\(i,j)=q}}(n-1-i)(n-j)\Bigr] \\
    &= 2\Bigl[2n^2-2(q+1)n+q+2\sum_{\substack{0<i,j<n-1\\(i,j)=q}}(n-1-i)(n-j)
    \end{aligned}\\
    +2\sum_{\substack{i=1\\(i,n-1)=q}}^{n-2}(n-1-i)\Bigr],
\end{multline}
\begin{align}
    f_q(n&-1) \notag \\
    &= 4\Bigl[(n-1)(n-1-q)+\sum_{\substack{0<i,j<n-1\\(i,j)=q}}(n-1-i)(n-1-j)\Bigr] \notag \\
    &= 4\Bigl[n^2-(q+2)n+q+1+\sum_{\substack{0<i,j<n-1\\(i,j)=q}}(n-1-i)(n-1-j)\Bigr], \label{fqproof3}
\end{align}
and
\begin{gather*}
    n^2-qn-[2n^2-2(q+1)n+q]+n^2-(q+2)n+q+1=1, \\
    (n-i)(n-j)-2(n-1-i)(n-j)+(n-1-i)(n-1-j)=1+i-j, \\
    n-i-(n-1-i)=1, \\
    \sum_{\substack{0<i,j<n-1\\(i,j)=q}}(1+i-j)
    = \sum_{\substack{0<i,j<n-1\\(i,j)=q}}1
    = 2\sum_{i=1}^{n-2}\sum_{\substack{j=1\\(i,j)=q}}^i1-1,
\end{gather*}
we have, by (\ref{fqproof1}), (\ref{fqproof2}), (\ref{fqproof3})
and~(\ref{sum1}),
\begin{align*}
    f_q(n)-{}&2f_q(n-1,n)+f_q(n-1) \\
    &= 4+4\Bigl(2\sum_{i=1}^{n-2}\sum_{\substack{j=1\\(i,j)=q}}^i 1-1\Bigr)
        +8\sum_{\substack{i=1\\(i,n-1)=q}}^{n-2}1 \\
    &= 8\sum_{i=1}^{n-2}\sum_{\substack{j=1\\(i,j)=q}}^i 1+8\sum_{\substack{i=1\\(i,n-1)=q}}^{n-2}1
    = 8\sum_{i=1}^{n-2}\sum_{\substack{j=1\\(i,j)=q}}^i 1+8\sum_{\substack{i=1\\(i,n-1)=q}}^{n-1}1 \\
    &= 8\sum_{i=1}^{n-2}e_q(i+1)+8e_q(n)
    = 8\sum_{i=2}^ne_q(i)=r_q(n),
\end{align*}
and (\ref{fq1}) follows.

To show (\ref{fq2}), we start from (\ref{fqproof2}), (\ref{fqproof3}),
and the fact that
\begin{multline*}
    f_q(n-2,n-1)
    = 2\Bigl[(n-2)(n-1-q)+(n-2-q)(n-1) \\
    \shoveright{+2\sum_{\substack{0<i<n-2\\0<j<n-1\\(i,j)=q}} (n-2-i)(n-1-j)\Bigr]} \\
    = 2\Bigl[2n^2-2(q+3)n+3q+4+2\sum_{\substack{0<i,j<n-1\\(i,j)=q}}(n-2-i)(n-1-j)\Bigr].
\end{multline*}
Since
\begin{multline*}
    2n^2-2(q+1)n+q-4[n^2-(q+2)n+q+1] \\
    \shoveright{+2n^2-2(q+3)n+3q+4=0,\quad} \\
    (n-1-i)(n-j)-2(n-1-i)(n-1-j)+(n-2-i)(n-1-j)=j-i,
\end{multline*}
we have
\begin{align*}
    f_q(n-1,n&)-2f_q(n-1)+f_q(n-2,n-1) \\
    &= 4\sum_{\substack{0<i,j<n-1\\(i,j)=q}}(j-i)+4\sum_{\substack{i=1\\(i,n-1)=q}}^{n-1}(n-1-i) \\
    &= 0+4\sum_{\substack{i=1\\(i,n-1)=q}}^{n-1}(n-1-i)
    = 2(n-1)e_q(n)
    = s_q(n).
\end{align*}

\medskip
\noindent \textit{Case 2.} $q=n-1$. Simply note that
\begin{multline*}
    f_{n-1}(n)-2f_{n-1}(n-1,n)+f_{n-1}(n-1) \\
    \shoveright{\begin{aligned}
    &= 4(n+1)-2\cdot2(n-1)+0 \\
    &= 8
    = 8e_{n-1}(n)
    = 8\sum_{i=2}^ne_{n-1}(i)=r_{n-1}(n),
    \end{aligned}\quad}\\
    \shoveleft{f_{n-1}(n-1,n)-2f_{n-1}(n-1)+f_{n-1}(n-2,n-1)} \\
    = 2(n-1)-0+0
    = 2(n-1)e_{n-1}(n)
    = s_{n-1}(n).
\end{multline*}
\end{proof}

\begin{corollary}
Let $n,q\ge 2$. Then
\begin{align}
    l_{\ge q}(n) &= 2l_{\ge q}(n-1,n)-l_{\ge q}(n-1)+\rho_{\ge q}(n), \label{lgeq1} \\
    l_{\ge q}(n-1,n) &= 2l_{\ge q}(n-1)-l_{\ge q}(n-2,n-1)+\sigma_{\ge q}(n). \label{lgeq2}
\end{align}
Here
\begin{equation*}
    \rho_{\ge q}(n)=4\sum_{i=2}^n(e_{q-1}(i)-e_q(i))
\end{equation*}
and so
\begin{equation*}
    \rho_{\ge q}(n)-\rho_{\ge q}(n-1)=4(e_{q-1}(n)-e_q(n)),
\end{equation*}
and
\begin{equation*}
    \sigma_{\ge q}(n)=(n-1)(e_{q-1}(n)-e_q(n)).
\end{equation*}
The initial values are $l_{\ge q}(n)=l_{\ge q}(n-1,n)=0$ for $n<q$, and
$l_{\ge q}(q-1,q)=q-1$.
\end{corollary}

\begin{proof}
Recall Corollary~\ref{explcor1}.
\end{proof}

\begin{corollary}
Let $n,q\ge 2$. Then
\begin{align}
    l_q(n) &= 2l_q(n-1,n)-l_q(n-1)+\rho_q(n), \label{lq1} \\
    l_q(n-1,n) &= 2l_q(n-1)-l_q(n-2,n-1)+\sigma_q(n). \label{lq2}
\end{align}
Here
\begin{equation*}
    \rho_q(n)=4\sum_{i=2}^n(e_{q-1}(i)-2e_q(i)+e_{q+1}(i))
\end{equation*}
and so
\begin{equation*}
    \rho_q(n)-\rho_q(n-1)=4(e_{q-1}(n)-2e_q(n)+e_{q+1}(n)),
\end{equation*}
and
\begin{equation*}
    \sigma_q­(n)=(n-1)(e_{q-1}(n)-2e_q(n)+e_{q+1}(n)).
\end{equation*}
The initial values are $l_q(n)=l_q(n-1,n)=0$ for $n<q$, and
$l_q(q-1,q)=q-1$.
\end{corollary}

\begin{proof}
Recall Corollary~\ref{explcor2}.
\end{proof}

Finally, we tie (\ref{fq1}) and (\ref{fq2}) together to obtain a single
recursive formula for~$f_q(n)$ only, and join similarly (\ref{lgeq1})
with (\ref{lgeq2}), and (\ref{lq1}) with (\ref{lq2}).

\begin{theorem}\label{recthm2}
Let $n,q\ge 2$. Then
\begin{align}
    f_q(n) &= f_q(n-1)+2\sum_{i=1}^ns_q(i)+2\sum_{i=1}^{n-1}r_q(i)+r_q(n), \label{fqthm3} \\
    l_{\ge q}(n) &= l_{\ge q}(n-1)+2\sum_{i=1}^n\sigma_{\ge q}(i)
        +2\sum_{i=1}^{n-1}\rho_{\ge q}(i)+\rho_{\ge q}(n), \notag \\
    l_q(n) &= l_q(n-1)+2\sum_{i=1}^n\sigma_q(i)+2\sum_{i=1}^{n-1}\rho_q(i)+\rho_q(n). \notag
\end{align}
\end{theorem}

\begin{proof}
The formula for~$l_{\ge 2}(n)$ has already been proved
\cite[Theorem~2]{EMHM}. The same proof applies to all the above claims.
\end{proof}

Because $c_{q+1}(n)=\frac{1}{2}f_q(n)$ by Theorem~\ref{explthm}, a
trivial modification of~(\ref{fq1}), (\ref{fq2}) and (\ref{fqthm3})
gives the corresponding formulas also for~$c_q(n)$.

\section{Remarks}\label{remarks}

\begin{remark}
A function~$d$, defined on~$G(n_1,n_2)$, is a (two-dimensional)
threshold function if it takes two values 0 and~1 and if there is a
line $a_1x_1+a_2x_2+b=0$ separating $d^{-1}(\{0\})$ and $d^{-1}(\{1\})$
(i.e., $d(x_1,x_2)=0\Leftrightarrow a_1x_1+a_2x_2+b\le 0$). Let
$t(n_1,n_2)$ denote the number of such functions.
Alekseyev~\cite[Theorem~3]{Al} proved (with different notation) that
$t(n_1,n_2)=f_1(n_1,n_2)+2$. So Theorems~\ref{recthm1} and
\ref{recthm2} also imply recursive formulas for $t(n)=t(n,n)$.
\end{remark}

\begin{remark}
If $k=2$ but $n_1\ne n_2$, the question about recursive formulas is
more difficult. Mustonen~\cite{Mu1} conjectured a recursive formula for
$l_{\ge 2}(n_1,n_2)$. If $k\ge 3$, this question remains open even in
case of $n_1=\dots=n_k$.
\end{remark}

\end{document}